\documentclass[tbtags,reqno,12pt]{amsart}

\usepackage{amsmath,amsthm,amsfonts}
\usepackage{youngtab,tikz}
\usepackage{subfigure}
\usepackage{tabmac_avi}

\newtheorem{theorem}{Theorem}[section]

\newtheorem{proposition}[theorem]{Proposition}

\newtheorem{example}[theorem]{Example}

\newcommand{\naturals}{\mathbb{N}}
\renewcommand{\P}{\mathbb{P}}
\DeclareMathOperator{\Des}{Des}
\DeclareMathOperator{\des}{des}
\DeclareMathOperator{\iDes}{iDes}

\DeclareMathOperator{\asc}{asc}
\DeclareMathOperator{\snk}{snk}
\newcommand{\mo}{\mathfrak{o}}

\author{Ryan Kaliszewski}
\title{Hook coefficients of chromatic functions}

\begin{document}
\maketitle

\begin{abstract}

The chromatic symmetric function of a graph is a generalization of the chromatic polynomial.  The key motivation for studying the structure of a chromatic symmetric function is to answer positivity conjectures by Stanley in 1995 and Gasharov in 1996.  As a symmetric function one can write the chromatic symmetric function in the basis of Schur functions.  In this paper we address the positivity of the Schur coefficients when the parameter of the Schur function is a hook shape.  Furthermore, we explore the hook coefficients of the chromatic quasisymmetric function introduced by Shareshian and Wachs in 2014 when expanded in the (Gessel) fundamental basis.
\end{abstract}


\section{Introduction}
\label{sec:in}

The chromatic symmetric function associated to a simple graph $G$ is introduced in \cite{Stanley95} as a generalization of the chromatic polynomial.  When the chromatic symmetric function is expanded in various bases of symmetric functions the coefficients are related to partitions of the vertices into stable subsets, the M\"{o}bius function of the lattice of contractions of $G$, and the structure of the acyclic orientations of $G$ \cite{Greene83,Stanley95,Whitney32}.  Also, for certain graphs these coefficients are related to Hecke algebras and Kazhdan Lusztig polynomials \cite{Haiman93,Stanley95}.  

Since the chromatic symmetric function of $G$, $\textbf{X}_G(\textbf{x})$, is indeed symmetric, it can be expanded in the basis of Schur functions, $\textbf{X}_G(\textbf{x})=\sum_\lambda c_\lambda s_\lambda$.  If we denote the hook $\lambda^{(k)}=(k,1,\ldots,1)$ we can refer to $c_{\lambda^{(k)}}$ as a \emph{hook coefficient}.  This paper will show that the hook coefficients are (positive) weighted enumerations of the acyclic orientations of $G$.  Specifically, we will prove the following theorem:

\begin{theorem}\label{introthm}
If $a_j$ denotes the number of acyclic orientations of the graph $G$ with $j$ sinks, then
\begin{equation} c_{\lambda^{(k)}}=\sum_{j=1}^n {j-1 \choose k-1}\cdot a_j \label{introthmline} \end{equation}
where $n$ is the number of vertices of $G$.
\end{theorem}

Previously Gasharov \cite{Gasharov96} has proven that if $G$ is the incomparability graph of a poset and $G$ is claw-free, then all of the coefficients in the Schur expansion, $c_\lambda$, are non-negative.  Gasharov then conjectured that for any graph that is claw-free all coefficients in the Schur expansion are non-negative.  Further, though in \cite{Chow97} formula \eqref{introthmline} was shown for incomparability graphs of (finite) posets, the author independently rediscovered the formula for \emph{any} graph by using Stanley's quasisymmetric construction in \cite{Stanley95}.

Also, in the case of the claw-free incomparability graph Gasharov has proven in \cite{Gasharov96} that the coefficient $c_\lambda$ is equal to the number of $P$-tableaux of shape $\lambda$.  There is a correspondence between $P$-tableaux of hook shapes and acyclic orientations of the claw-free incomparability graph presented in \cite{Chow97}.  In this paper we will show that the correspondence holds even when the graph is not claw-free.

To prove formula \eqref{introthmline} we will consider the quasisymmetric refinement of the chromatic symmetric function that was introduced by Shareshian and Wachs in \cite{Shareshian14}.  Since this function is quasisymmetric it can be expanded in the basis of (Gessel) fundamental quasisymmetric functions $\textbf{X}_G(\textbf{x},t)=\sum_\alpha d_\alpha(t) F_\alpha$.  We will discuss a formula for the hook coefficients $d_{\lambda^{(k)}}(t)$ of $\textbf{X}_G(\textbf{x},t)$ and show that it specializes to forumla \eqref{introthmline} in the symmetric case.

\section{Background definitions}

The ring of symmetric functions over $\mathbb{Z}$ is 
$$\Lambda=\mathbb Z[h_1,h_2,\ldots,]\,,$$
where $h_r=\sum_{1\leq i_1\leq \cdots\leq i_r} x_{i_1}\cdots x_{i_r}$.
Bases for the space $\Lambda$ are traditionally indexed by
{\it partitions} $\lambda=(\lambda_1,\ldots,\lambda_\ell)\in \mathbb Z_{>0}^\ell$ 
where $\lambda_1\geq\cdots\geq \lambda_\ell$.  
For example, 
the monomial symmetric functions $m_\lambda$ are defined for each
partition $\lambda$ as the sum over $x^\alpha$ for
each distinct rearrangement $\alpha$ of the parts of $\lambda$.

It is often convenient to view a partition as a diagram.
In particular, the {\it Ferrers diagram} with a partition shape, $\lambda$, 
is the configuration of boxes stacked in a left-justified manner
with $\lambda_i$ boxes in row $i$, for $1\leq i\leq\ell$.
We use the French notation, where row 1 is the largest and
bottommost row.
\begin{example}
The partition $\lambda=(4,4,3)$ corresponds to the diagram
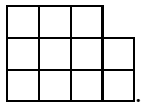
\[
\footnotesize
\tableau[sbY]
{,, \cr ,,, \cr ,,,\cr}.
\qquad
\qquad
\]

\end{example}
A {\it skew-shape} is a pair of partitions $\lambda=(\lambda_1,\ldots,\lambda_\ell)$ and $\mu=(\mu_1,\ldots,\mu_r)$ where $r\leq \ell$ and $\mu_i\leq \lambda_i$, for each $i=1,\ldots, r$.  The skew-shape is denoted $\lambda/\mu$ and the diagram is the set-theoretic difference between the diagrams of $\lambda$ and $\mu$.  A {\it horizontal strip} is a skew-shape where each column contains at most one cell.

\begin{example}
The skew shape $(4,3,1)/(3,2)$ is a horizontal strip,
\[(4,3,1)/(3,2)=
\footnotesize
\tableau[sbY]
{ \cr \sk,\sk, \cr \sk,\sk,\sk, \cr}.
\qquad
\qquad
\]
\end{example}

Our focus is on the Schur function basis which may
be defined in many ways.  We consider its characterization 
as the generating functions of {\it tableaux}.  A semi-standard
tableau  of weight $\mu=(\mu_1,\ldots,\mu_r)$ is a nested 
sequence of partitions 
\begin{equation}
\label{tab}
\emptyset=
\lambda^{(0)}\subset\lambda^{(1)}\subset\cdots \subset\lambda^{(r)}
\end{equation}
such that
$\lambda^{(i)}/\lambda^{(i-1)}$ is a horizontal $\mu_i$-strip.
It is generally represented with a filling of shape $\lambda^{(r)}$
by placing $i$ in the cells of the skew $\lambda^{(i)}/\lambda^{(i-1)}$.
When the weight of a tableau is $(1,1,\ldots,1)$ it is called 
{\it standard}.  $SSYT(\lambda,\mu)$ denotes 
the set of semi-standard tableaux of shape $\lambda$ and weight $\mu$
and the union over all weights is $SSYT(\lambda)$.
For any partition $\lambda$, the {\it Schur function}, $s_\lambda$,
can be defined by
$$
s_\lambda = \sum_{T\in SSYT(\lambda)} x^{{\rm weight}(T)}\,.
$$

We shall denote the set $SYTT(\lambda,1^n)$ of standard tableaux of shape $\lambda$
by $SYT(\lambda)$.
Recall that if $T\in SYT(\lambda)$ then the row reading word, 
which we will denote $r(T)$, is an ordered list of positive integers 
obtained by listing the entries of the tableau starting from the upper 
left corner and proceeding left to right, top to bottom.  
Observe that the row reading word of any standard Young tableau 
of shape $\lambda$ is necessarily a permutation, $r(T)\in\mathfrak{S}_n$ 
where $\lambda\vdash n$. 
\begin{example}
The set of all standard tableaux on 5 letters with shape $\lambda=(3,2)$ is
\[
\footnotesize
T_1=\tableau[sbY]
{ 4,5 \cr 1,2,3 \cr},
\qquad
\qquad
T_2=\tableau[sbY]
{ 3,5 \cr 1,2,4 \cr},
\qquad
\qquad
T_3=\tableau[sbY]
{ 3,4 \cr 1,2,5 \cr},
\qquad
\qquad
\]
\[
\footnotesize
T_4=\tableau[sbY]
{ 2,5 \cr 1,3,4 \cr},
\qquad
\qquad
T_5=\tableau[sbY]
{ 2,4 \cr 1,3,5 \cr}.
\qquad
\qquad
\]
The respective reading words are $r(T_1)=45123, r(T_2)=35124, r(T_3)=34125, r(T_4)=25134,$ and $r(T_5)=24135.$

\end{example}
More details on symmetric functions and tableaux can be found
in, for example, \cite{Stanley99,Fulton97,Macdonald95}.

Another useful tool in the study of symmetric function theory is the set of
quasisymmetric functions.   Instead of partitions, these are
associated to permutations.  For any  $\sigma\in\mathfrak{S}_n$,
define the \emph{descent set}, $\Des(\sigma)$, 
to be the set $\{i\in [n-1]|\sigma(i+1)<\sigma(i)\}$.  If $\Des(\sigma)=\{i_1<i_2<\ldots<i_k\}$ 
define the composition $\Des'(\sigma)=(i_1,i_2-i_1,i_3-i_2,\ldots,n-i_k)\models n$.  We will use a lower case to be the cardinality and $\iDes$ to denote the descent set of the inverse, i.e. 
\[\iDes(\sigma)=\Des(\sigma^{-1}) \hspace{5mm} \text{ and } \hspace{5mm} \des(\sigma)=|\Des(\sigma)|.\]

For any permutation $\sigma\in\mathfrak{S}_n$,
the {\it fundamental quasisymmetric function} is
defined as
\begin{equation} 
F_{\Des(\sigma)}=
\sum_{\substack{i_1\leq\ldots\leq i_n\\i_j=i_{j+1}\Rightarrow j\notin \Des(\sigma)}}
x_{i_1}\cdots x_{i_n}.\end{equation}
Note by the correspondence outlined above, we can alternatively
index fundamental quasisymmetric functions by compositions.  

In \cite{Gessel84} Gessel proved that, for any partition $\lambda$,
the Schur function is
\begin{equation}
\label{eq:gessel}
 s_\lambda=\sum_{T\in SYT(\lambda)} F_{\iDes(r(T))}. 
\end{equation}
Although it is not immediately clear from this definition 
that a Schur function is symmetric, a combinatorial proof 
is given in \cite{Bender72}.

\section{Chromatic Symmetric Functions}
Throughout this section let $G=(V,E)$ be a graph with finite vertex set $V$ and finite edge set $E$.  We will also assume that $G$ is a {\it simple graph}---a graph that is simply laced and has no loops.

A {\it proper coloring} of a graph $G$ is a map $\kappa:V\rightarrow\P$ such that for all $\{u,v\}\in E, \kappa(u)\neq \kappa(v)$.  That is, adjacent vertices never share the same color.  Let $\mathcal{C}(G)$ be the set of proper colorings of $G$.  In \cite{Stanley95} Stanley defined the {\it chromatic symmetric function} of $G$ to be
\[ \textbf{X}_G(\textbf{x}):=\sum_{\kappa\in\mathcal{C}(G)} \textbf{x}_\kappa, \]
where $\textbf{x}:=(x_1,x_2,\ldots)$ is a sequence of commuting indeterminates and $\textbf{x}_\kappa=\prod_{v\in V}x_{\kappa(v)}$.

Observe that for any coloring, $\kappa$, and $\sigma$ any permutation of $\P$, $\sigma\circ\kappa$ is also a coloring of $G$.  Thus this shows that $\textbf{X}_G$ is indeed symmetric.  Furthermore, it is immediate that if $|V|=n$ then $\textbf{X}_G$ is homogeneous of degree $n$.

A {\it hook} is a partition of the following form:  $(k,1,1,\ldots,1)$, where $k\in\P$.  When we expand $\textbf{X}_G$ in the basis of Schur functions a {\it hook coefficient} is any coefficient of a Schur function with a hook shape.

\begin{example}
The chromatic symmetric function of the claw graph is  
\begin{equation} \textbf{X}_{K_{3,1}}=s_{(3,1)}-s_{(2,2)}+5s_{(2,1,1)}+8s_{(1,1,1,1)}.  \end{equation}
The hook coefficients are 1 for the hook $(3,1)$, 5 for the hook $(2,1,1)$, and 8 for the (trivial) hook $(1,1,1,1)$.
\end{example}
\begin{example}
The chromatic symmetric function of the complement of $K_3$ is
\begin{equation} \textbf{X}_{\overline{K_3}}=s_{(3)}+2s_{(2,1)}+s_{(1,1,1)}. \end{equation}
The hook coefficients are 1 for the hook $(3)$, 2 for the hook $(2,1)$, and 1 for the hook $(1,1,1)$.
\end{example}

There are some interesting known properties of the chromatic symmetric function.  First of all, it is a true generalization of the chromatic function of a graph.   Recall that the {\it chromatic function} of a graph, $G$, is the unique function, $\chi_G(x)$ of degree $n=|V|$ such that $\chi_G(k)$ is the number of proper colorings of a graph using no more than $k$ colors.

\begin{example}
The chromatic function for the path graph, $P_n$ is 
\begin{equation} \chi_{P_n}(x)=x(x-1)^{n-1}. \end{equation}

Specifically for $P_3$:  \begin{tikzpicture}[node distance=1.5cm]
  \node[draw, circle] (1) {};
  \node[draw, circle, right of = 1] (2) {};
  \node[draw, circle, right of = 2] (3) {};

  \draw[thick]	(2) -- (1);
  \draw[thick]	(2) -- (3);
\end{tikzpicture},
\begin{align*} \chi_{P_3}(0)&=0, \text{ true for all chromatic functions, }\\
\chi_{P_3}(1)&=0, \text{ there is no way to color $P_3$ with just one color, } \\
\chi_{P_3}(2)&=2, \text{ 1-2-1 and 2-1-2, and }\\
\chi_{P_3}(3)&=12, \text{ which includes the 2-colorings from above.}
\end{align*}
It is easy to check that $\chi_{P_3}(x)=x(x-1)^2$.
\end{example}

\begin{proposition} [Stanley, 1995]

Let $w_k\in\naturals^n$ be such that
\begin{equation} (w_k)_i=\left\{\begin{array}{ll}1 & \text{if } i\leq k \\ 0& \text{otherwise}\end{array}\right. . \end{equation}
Then $\textbf{X}_G(w_k)=\chi_G(k)$.
\end{proposition}

An {\it orientation} of a graph is an association of an arrow to each edge.  An orientation is {\it acyclic} if there are no cycles in the resulting oriented graph.  In an oriented graph, a vertex is called a {\it sink} if there are no arrows pointing away from the vertex.  Note that an isolated vertex is considered a sink.

If one writes $\textbf{X}_G$ in terms of the elementary symmetric functions:
\[ \textbf{X}_G(\textbf{x})=\sum_{\lambda\vdash n} b_\lambda e_\lambda, \]
where the sum is over all partitions, then the coefficients have the following property.

\begin{proposition} [Stanley, 1995]

Let $b_\lambda$ be defined as above.  Let $a_k(G)$ be the number of acyclic orientations of $G$ with $k$ sinks.  Then
\[ a_k(G)=\sum_{\substack{\lambda\vdash n\\ \ell(\lambda)=k}}b_\lambda. \]
\end{proposition}

\section{Chromatic Quasisymmetric Functions}
Let $G=(V,E)$ be a simple graph with finite vertex set.  Suppose that $[V]=\{1,2,\ldots,|V|\}$ and $\zeta$ is a labeling of the vertex set, {\it i.e.} $\zeta:[V]\rightarrow V$ is a bijection.  Let $\textbf{x}=(x_1,x_2,\ldots)$ be a list of commuting indeterminates.  In \cite{Shareshian14} Shareshian and Wachs define the chromatic quasisymmetric function of $G$ with respect to the labeling $\zeta$ to be
\[ \textbf{X}_{G,\zeta}(\textbf{x},t)=\sum_{\kappa\in\mathcal{C}(G)}t^{\asc(\kappa)}\textbf{x}_\kappa, \]
where 
\[ \asc(\kappa):=|\{\{u,v\}\in E | \zeta(u)<\zeta(v) \text{ and } \kappa(u)<\kappa(v)\}|. \] 
Note that the definition of $\textbf{X}_{G,\zeta}(\textbf{x},t)$ relies on both the isomorphism class of the graph and the choice of labeling, unlike  $\textbf{X}_G(\textbf{x})$, which only relies on the isomorphism class of $G$.  However, it is immediate that  $\textbf{X}_{G,\zeta}(\textbf{x},1)= \textbf{X}_G(\textbf{x})$ for any $\zeta$.

Let $O(G)$ be the set of acyclic orientations of the graph $G$.  For each acyclic orientation $\mathfrak{o}$ we define two statistics, the descents and the sinks.  Let $\des(\mo)$ be the the number of edges that agree with the vertex labeling, {\it i.e.} $(u,v)$ is a directed edge of $\mo$ and $\zeta(u)>\zeta(v)$.  Let $\snk(\mo)$ be the number of sinks---vertices in which there are no arrows based at that vertex.  Note that an isolated vertex is a sink.

Recall that when we expanded the chromatic symmetric function in the fundamental quasisymmetric basis, we labeled the coefficients $\textbf{X}_G(\textbf{x},t)=\sum_\alpha d_\alpha(t)F_\alpha.$  This brings us to:

\begin{theorem} \label{bigtheorem}When the chromatic quasisymmetric function associated to $G$ and $\zeta$ is expanded in the basis of (Gessel) fundamental quasisymmetric functions the hook coefficients are
\[ d_{\lambda^{(k)}}(t)=\sum_{\mo\in O(G)} {\snk(\mo)-1 \choose k-1}\cdot  t^{\des(\mo)}. \]
\end{theorem}

\begin{proof}
We will be following the dual of the notation from \cite{Shareshian14}, that is we are inverting all of the arrows so that we get results in terms a sinks rather than sources.  For each $\mo\in O(G)$ select $\omega_\mo$ to be an {\it increasing} labeling such that the sinks are minimal.  That is, for all vertices $u,v$, if there is a directed path from $u$ to $v$ then $\omega_\mo(u)>\omega_\mo(v)$, if $\snk(\mo)=s$ and if $v$ is a sink, then $\omega_\mo(v)\leq s$.  

Set $\mathcal{L}'(\mo,\omega_\mo)$ to be the set of all linear extensions dual to the labeled poset $(\mo,\omega_\mo)$.  One can express $\mathcal{L}'(\mo,\omega_\mo)$ as a set of permutations $[\omega_\mo(v_1),\ldots,\omega_\mo(v_n)]$ (written in one line notation) where $v_1,\ldots,v_{|V|}$ is a sequencing of the vertices such that if there is a directed path in $\mo$ from vertex $v_i$ to vertex $v_j$ then $i<j$.

For example, if we consider the path graph on four vertices, with orientation and increasing labeling, $\omega$, as:

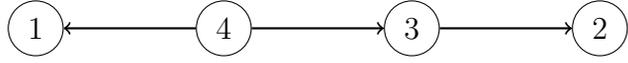
\begin{center}\begin{tikzpicture}[node distance=2.5cm]
  \node[draw, circle] (1) {1};
  \node[draw, circle, right of = 1] (2) {4};
  \node[draw, circle, right of = 2] (3) {3};
  \node[draw, circle, right of = 3] (4) {2};

  \draw[thick,->]	(2) -- (1);
  \draw[thick,->]	(2) -- (3);
  \draw[thick,->]	(3) -- (4);
\end{tikzpicture},\end{center}

then the sequencings $v_1,\ldots,v_4$ would be

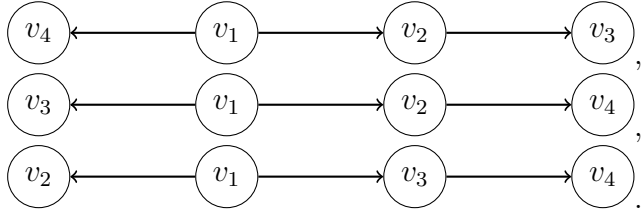
\begin{center}\begin{tikzpicture}[node distance=2.5cm]
  \node[draw, circle] (1) {$v_4$};
  \node[draw, circle, right of = 1] (2) {$v_1$};
  \node[draw, circle, right of = 2] (3) {$v_2$};
  \node[draw, circle, right of = 3] (4) {$v_3$};

  \draw[thick,->]	(2) -- (1);
  \draw[thick,->]	(2) -- (3);
  \draw[thick,->]	(3) -- (4);
\end{tikzpicture},

\begin{tikzpicture}[node distance=2.5cm]
  \node[draw, circle] (1) {$v_3$};
  \node[draw, circle, right of = 1] (2) {$v_1$};
  \node[draw, circle, right of = 2] (3) {$v_2$};
  \node[draw, circle, right of = 3] (4) {$v_4$};

  \draw[thick,->]	(2) -- (1);
  \draw[thick,->]	(2) -- (3);
  \draw[thick,->]	(3) -- (4);
\end{tikzpicture},

\begin{tikzpicture}[node distance=2.5cm]
  \node[draw, circle] (1) {$v_2$};
  \node[draw, circle, right of = 1] (2) {$v_1$};
  \node[draw, circle, right of = 2] (3) {$v_3$};
  \node[draw, circle, right of = 3] (4) {$v_4$};

  \draw[thick,->]	(2) -- (1);
  \draw[thick,->]	(2) -- (3);
  \draw[thick,->]	(3) -- (4);
\end{tikzpicture}.

This means that the respective elements of $\mathcal{L}'(\mo,\omega_\mo)$ would be:
\[ 4321, \hspace{5mm} 4312, \hspace{5mm} \text{and} \hspace{5mm} 4132.\]
\end{center}

From equations (3.3) and (3.4) of \cite{Shareshian14} we have for {\it any} simple graph, $G$:
\begin{equation} \textbf{X}_{G,\zeta}(\textbf{x},t)=\sum_{\mo\in O(G)} t^{\des(\mo)} \sum_{\sigma\in\mathcal{L}'(\mo,\omega_\mo)} F_{n-\Des(\sigma)}, \end{equation}
where $n-\Des(\sigma)=\{i|n-i\in \Des(\sigma)\}.$
If we expand $\textbf{X}_{G,\zeta}(\textbf{x},t)$ in the basis of fundamental quasisymmetric functions this gives us:
\begin{equation} \sum_{\alpha}d_\alpha(t)\cdot F_\alpha = \sum_{\mo\in O(G)} t^{\des(\mo)} \sum_{\sigma\in\mathcal{L}'(\mo,\omega_\mo)} F_{n-\Des(\sigma)}. \end{equation}
Since the fundamental quasisymmetric functions form a basis and the partition (composition) $\lambda^{(k)}$ corresponds to the descent set $\{k,k+1,\ldots, n-1\}$ we get
\begin{equation} d_{\lambda^{(k)}}(t) = \sum_{\mo\in O(G)} t^{\des(\mo)} \cdot\sum_{\substack{\sigma\in\mathcal{L}'(\mo,\omega_\mo)\\ \Des(\sigma)=\{1,2,\ldots,n-k\}}} 1. \end{equation}
So all that is left to show is that the number of permutations $\sigma\in\mathcal{L}'(\mo,\omega_\mo)$ that have descent set $\{1,2,\ldots,n-k\}$ is ${\snk(\mo)-1 \choose k-1}$.

First we will address the structure of the permutations.  In order to have descent set $\{1,2,\ldots,n-k\}$ the permutation $\sigma$ must be decreasing in its first $n-k$ entries and then increasing thereafter.  In order for this to be true there {\it must} be a 1 in the $(n+1-k)^{th}$ location and all entries following the 1 must correspond to sinks.  If we suppose otherwise, that $u$ is not a sink and $\omega_\mo(u)$ follows the 1 entry, then note that there is a directed path from $u$ to a sink, $v$.  However, due to the minimality of labeling of sinks, we know that $\omega_\mo(v)<\omega_\mo(u)$.  But in the definition of $\sigma,$ recall that $\omega_\mo(u)$ must precede $\omega_\mo(v)$ which creates a descent after position $n-k$, which is a contradiction.

Thus by fixing the position of 1 within $\sigma$ and choosing $k-1$ sinks whose labels follow the 1 entry we have completely determined the permutation $\sigma$.  Note that any choice of $k-1$ sinks works because there are no directed paths between the sinks.  Furthermore, the labeling completely determines the remainder of $\sigma$ after the choice of the `following' sinks.  Thus there are ${\snk(\mo)-1 \choose k-1}$ such permutations.  This gives us that
\begin{equation} \sum_{\substack{\sigma\in\mathcal{L}'(\mo,\omega_\mo)\\ \Des(\sigma)=\{1,2,\ldots,n-k\}}} 1 = {\snk(\mo)-1 \choose k-1}, \end{equation}
which completes the proof.
\end{proof}

When we consider the symmetric case, we substitute $t=1$ into theorem \ref{bigtheorem}.  We can then group acyclic orientations together by their numbers of sinks.  This gives us the following corollary:

\begin{theorem}[Proof of Theorem \ref{introthm}] \label{maintheorem}
The hook coefficients of the chromatic symmetric function are given by
\begin{equation} c_{\lambda^{(k)}}=\sum_{j=1}^{|V|} {j-1 \choose k-1} \cdot a_j \end{equation}
where $a_j$ is the number of acyclic orientations of $G$ with $j$ sinks.
\end{theorem}

\begin{proof}
Theorem 15 of \cite{Egge10} shows that $c_{\lambda^{(k)}}=d_{\lambda^{(k)}}=d_{\lambda^{(k)}}(1)$.  From theorem \ref{bigtheorem} we have:
\begin{equation} d_{\lambda^{(k)}}(1)=\sum_{\mo\in O(G)} {\snk(\mo)-1 \choose k-1}=\sum_{j=1}^{|V|} {j-1 \choose k-1} \cdot a_j, \end{equation}
where $a_j$ is the number of acyclic orientations with $j$ sinks.  This completes the proof.
\end{proof}

Consider the set of $P$-tableaux that correspond to an acyclic orientation, $\mathfrak{o}$ in \cite{Chow97}.  From Lemma 1 in this paper we know that every $P$-tableaux corresponds to some acyclic orientation and acyclic orientations correspond to ${j-1 \choose k-1}$ $P$-tableaux (the construction can be followed even if the graph is not claw-free).  By Chow's result there are 
\begin{equation} \sum_{j=1}^{|V|} {j-1 \choose k-1} \cdot a_j \end{equation}
$P$-tableaux of hook-shape, $\lambda^{(k)}$, which is equal to the coefficient in Theorem \ref{maintheorem}.  This leads us to:

\begin{proposition}
Let $G$ be the incomparability graph of a poset $P$.  The hook coefficient $c_{\lambda^{(k)}}$ is equal to the number of $P$-tableaux of shape $\lambda^{(k)}$, \emph{regardless of whether $G$ is claw-free or not}.
\end{proposition}

\bibliographystyle{alpha}
\bibliography{MyBibliography}

\end{document}